\documentclass[12pt]{article}
\usepackage{amssymb,amsmath}
\usepackage{amsthm}
\usepackage{cases}
\usepackage{amsfonts}
\usepackage{graphicx}
\usepackage{cite,color,xcolor}
\usepackage[margin=1 in]{geometry}
\usepackage[colorlinks,citecolor=blue,urlcolor=blue]{hyperref}
\usepackage[utf8]{inputenc}

\newtheorem{theorem}{Theorem}[section]

\newtheorem{lemma}{Lemma}[section]
\newtheorem{proposition}{Proposition}[section]

\newtheorem{definition}{Definition}[section]
\newtheorem{remark}{Remark}[section]

\newcommand{\bal}{\begin{align}}
\newcommand{\bbal}{\begin{align*}}
\newcommand{\beq}{\begin{equation}}
\newcommand{\eeq}{\end{equation}}
\newcommand{\bca}{\begin{cases}}
\newcommand{\eca}{\end{cases}}
\def\div{\mathord{{\rm div}}}
\newcommand{\pa}{\partial}
\newcommand{\fr}{\frac}
\newcommand{\na}{\nabla}
\newcommand{\De}{\Delta}

\newcommand{\cd}{\cdot}
\newcommand{\ep}{\varepsilon}
\newcommand{\dd}{\mathrm{d}}

\newcommand{\R}{\mathbb{R}}

\newcommand{\les}{\lesssim}

\newcommand{\f}{\left}
\newcommand{\g}{\right}

\begin{document}
\bibliographystyle{plain}
\title{On the continuous properties for the 3D incompressible rotating Euler equations}

\author{Jinlu Li$^{1}$,
Yanghai Yu$^{2}$\footnote{E-mail: lijinlu@gnnu.edu.cn; yuyanghai214@sina.com(Corresponding author); zhuneng@ncu.edu.cn} and Neng Zhu$^{3}$\\
\small $^1$ School of Mathematics and Computer Sciences, Gannan Normal University, Ganzhou 341000, China\\
\small $^2$ School of Mathematics and Statistics, Anhui Normal University, Wuhu 241002, China\\
\small $^3$ Department of Mathematics, Nanchang University, Nanchang 330031, China }

\date{\today}

\maketitle\noindent{\hrulefill}

{\bf Abstract:} In this paper, we consider the Cauchy problem for the 3D Euler equations with the Coriolis force in the whole space. We first establish the local-in-time existence and uniqueness of solution to this system in $B^s_{p,r}(\R^3)$. Then we prove that the Cauchy problem is ill-posed in two different sense:
(1) the solution of this system is not uniformly continuous dependence on the initial data in the same Besov spaces, which extends the recent work of Himonas-Misio{\l}ek \cite[Comm. Math. Phys., 296, 2010]{HM1} to the more general framework of Besov spaces;
(2) the solution of this system cannot be H\"{o}lder continuous in time variable in the same Besov spaces. In particular, the solution of the system is discontinuous in the weaker Besov spaces at time zero. To the best of our knowledge, our work is the first one addressing the issue on the failure of H\"{o}lder continuous in time of solution to the classical Euler equations with(out) the Coriolis force.


\vskip0mm\noindent{\hrulefill}

\section{Introduction}

In this paper, we focus on the Cauchy problem for the Euler equations with the Coriolis force in $\R^3$,
describing the motion of perfect incompressible fluids in the rotational framework:
\begin{align}\label{CE}
\begin{cases}
\pa_t u+\Omega e_3\times u+u\cdot \nabla u+\nabla P=0, &\quad (t,x)\in \R^+\times\R^3,\\
\mathrm{div\,} u=0,&\quad (t,x)\in \R^+\times\R^3,\\
u(t=0,x)=u_0(x), &\quad x\in \R^3,
\end{cases}
\end{align}
where the vector field $u(t,x):[0,\infty)\times {\mathbb R}^3\to {\mathbb R}^3$ stands for the velocity of the fluid, the quantity $P(t,x):[0,\infty)\times {\mathbb R}^3\to {\mathbb R}$ denotes the scalar pressure, and $\mathrm{div\,} u=0$ means that the fluid is incompressible.
The constant $\Omega\in\R$ represents the speed of rotation around
the vertical unit vector $e_3 = (0, 0, 1)$ and is called the Coriolis parameter. Problems concerning large-scale atmospheric and oceanic flows are known to be dominated by rotational effects. For this reason, almost all of the models of oceanography and meteorology dealing with large-scale phenomena include a Coriolis force. For example, an oceanic circulation featuring a hurricane is caused by the large rotation. There is no doubt that other physical effects are of similar significance like salinity, natural boundary conditions and so on.  The dispersive effect of rotation in fluid flows has been studied in the literature from various perspectives, see, e.g., geophysical flows \cite{IR,JM,JP}, life span and asymptotic behaviour in the case of fast rotations \cite{ch1,ch2,d,KY} and almost global stability \cite{YGUO}.

Before stating related results for \eqref{CE} precisely, we recall the notion of well-posedness in the sense of Hadamard. We say that the Cauchy problem
\begin{equation}\label{dp}
\begin{cases}
\pa_tf=F(f), \\
f(0,x)=u_0(x),
\end{cases}
\end{equation}
is locally Hadamard well-posed in a Banach space $X$ if the following conditions hold
\begin{enumerate}
  \item (Local existence)\; For any initial data $u_0\in X$, there exists a short time $T = T(u_0) > 0$ and a solution $\mathbf{S}_{t}(u_0)\in\mathcal{C}([0,T),X)$ to the Cauchy problem \eqref{dp};
  \item (Uniqueness)\; This solution $\mathbf{S}_{t}(u_0)$ is unique in the space $\mathcal{C}([0,T),X)$;
  \item (Continuous Dependence)\; The data-to-solution map $u_0 \mapsto \mathbf{S}_{t}(u_0)$ is continuous in the following
sense:  for any $u_0$, any $T_1< T\left(u_0\right)$ and any $\varepsilon>0$, there exists $\delta=\delta(\varepsilon,u_0)$ such that for any $\widetilde{u}_0$, if $\left\|u_0-\widetilde{u}_0\right\|_X \leq \delta$ then $\mathbf{S}_{t}(\widetilde{u}_0)$ exists up to $T_1$ and
$$\|\mathbf{S}_{t}(u_0)-\mathbf{S}_{t}(\widetilde{u}_0)\|_{\mathcal{C}([0,T_1],X)}\leq \varepsilon.$$
\end{enumerate}
Furthermore, we say that the data-to-solution map $u_0 \mapsto \mathbf{S}_{t}(u_0)$ is uniform continuous if for any $T_1 < T$ and any $\varepsilon > 0$, there exists $ \delta=\delta(\varepsilon)> 0$ independent of $u_0$ or $\widetilde{u}_0$ such that, if $\|u_0-\widetilde{u}_0\|_{X}\leq \delta$,
then $\mathbf{S}_{t}(\widetilde{u}_0)$ exists up to $T_1$ and
$$\|\mathbf{S}_{t}(u_0)-\mathbf{S}_{t}(\widetilde{u}_0)\|_{\mathcal{C}([0,T_1],X)}\leq \varepsilon.$$
We would like to emphasize that, in this paper, we say the Cauchy problem \eqref{dp} is {\bf well-posed} in a Banach space $X$ if it enjoys existence, uniqueness and continuous dependence on the initial data. We say that the Cauchy problem \eqref{dp} is {\bf semi-linearly well-posed} (see \cite{NT}) for the initial data in $X$ if it is Hadamard well-posed and in addition, the solution map $u_{0}\mapsto \mathbf{S}_{t}(u_0)$ is uniformly continuous from any bounded subset $B\subset X$ into $C\f([0,T];X\g)$. In particular, we are interested in the lack of continuity with respect to time variable or uniform continuity of solution with respect to the initial data.

\subsection{Known Well/Ill-posedness results}
Let us first review some of the previous works in this direction.

{\bf The classical Euler equations.}\;
In the case of $\Omega=0$, \eqref{CE} reduces to the original Euler equations. Kato \cite{Kato} proved the local well-posedness of classical solution to Euler equations in the Sobolev space $H^s(\mathbb{R}^3)$ for all $s>5/2$. Kato-Ponce \cite{KatoP} extended this result to the Sobolev spaces $W^{s, p}(\mathbb{R}^3)$ of fractional order for $s>3 / p+1,1<p<\infty$. Chae \cite{Chae1,Chae2,Chae3} and Chen-Miao-Zhang \cite{Miao} gave further extensions to the Triebel-Lizorkin spaces $F_{p, q}^s(\mathbb{R}^3)$ with $s>3 / p+1,1<p, q<\infty$ and the Besov spaces $B_{p, q}^s(\mathbb{R}^3)$ with $s>3 / p+1$, $1<p<\infty, 1 \leq q \leq \infty$ or $s=3 / p+1,1<p<\infty, q=1$. However, these two kinds of function spaces are only in the $L^p(1<p<\infty)$-framework since the Riesz transform is not bounded on $L^{\infty}$. The currently-known best result on the local existence was given by Pak-Park \cite{Pak} in the Besov space $B_{\infty, 1}^1(\mathbb{R}^3)$.  Guo-Li-Yin \cite{GLY} proved the continuous dependence of the  Euler equations in the space $B_{p, q}^s(\mathbb{R}^3)$ with $s>3 / p+1$, $1\leq p\leq \infty, 1 \leq q < \infty$ or $s=3 / p+1,1\leq p\leq \infty, q=1$. Later, Cheskidov-Shvydkoy \cite{Cheskidov} proved that the solution of the Euler equations cannot be continuous as a function
of the time variable at $t = 0$ in the spaces $B^s_{r,\infty}(\mathbb{T}^d)$ where $s > 0$ if $2 < r \leq \infty$ and $s>d(2/r-1)$
if $1 \leq r \leq 2$. Furthermore, Bourgain-Li \cite{Bourgain1,Bourgain2} proved the strong local ill-posedness of the Euler equations in borderline Besov spaces $B^{d/p+1}_{p,r}$ with $(p,r)\in[1,\infty)\times(1,\infty]$ when $d=2,3$. Subsequently, Misio{\l}ek-Yoneda \cite{MY} studied the borderline cases and showed that the 2D Euler equations are not locally well-posed in the sense of Hardamard in the $C^1$ space and in the Besov space $B^1_{\infty,1}$.
Recently, Misio{\l}ek-Yoneda \cite{MY2} showed that the solution map for the Euler equations is not even continuous in the space of H\"{o}lder continuous functions and thus not locally Hadamard well-posed in $B^{1+s}_{\infty,\infty}$ with any $s\in(0,1)$. Concerning the non-uniform continuity of the data-to-solution map, we would like to mention that the beautiful results of Himonas-Misio{\l}ek \cite{HM1} covered both the torus $\mathbb{T}^d$ and the whole spaces $\R^d$ cases. More precisely, they proved that the solution map for the Euler equations in bi(tri)-dimension is not uniformly continuous in Sobolev spaces $H^s(\mathbb{T}^d)$ for $s\in \R$ and in $H^s(\mathbb{R}^d)$ for any $s>0$. Bourgain and Li \cite{Bou} settled the border line case $s = 0$.

{\bf The Euler equations with the Coriolis force.}\;
In the case of $\Omega\neq0$, as it is already observed in \cite{ch1,ch2,d}, the Euler-Coriolis system \eqref{CE} exhibits a dispersion phenomenon which is due to the presence of the Coriolis force $\Omega e_3\times u$. For large Coriolis parameter $|\Omega|$, Dutrifoy \cite{d} showed the asymptotics of solutions to vortex
patches or Yudovich solutions as the Rossby number goes to zero for some particular initial data,
and gave a lower bound on the lifespan $T_\Omega$ of the solution to \eqref{CE} as $T_\Omega\gtrsim \log\log|\Omega|$. For $u_0 \in H^s(\mathbb{R}^3)$ with $s>5/2$, Koh-Lee-Takada \cite{KLT} proved that there exists a unique local in time solution $u$ for \eqref{CE} with $\Omega\in\R$ in the class $\mathcal{C}([0, T]; H^s(\mathbb{R}^3)) \cap \mathcal{C}^1([0, T] ; H^{s-1}(\mathbb{R}^3))$ (see also \cite{T,WC}). Moreover, assuming that $s>7/2$, they showed that their solutions can be extended to long-time intervals $\left[0, T_{\Omega}\right]$ provided that the speed of rotation is large enough. V. Angulo-Castillo and L.C.F Ferreira \cite{CF} further gave extensions to the critical Besov space $B^{5/2}_{2,1}(\R^3)$. Jia-Wan \cite{JW} proved the long time existence of classical solutions to \eqref{CE} for initial data in the
Sobolev space $H^s(s > 5/2)$  with a weaker assumption on the lower bound of $|\Omega|$. However, the above mentioned results for \eqref{CE} only involve the short-time existence, uniqueness for any speed of rotation and long-time existence for high speed of rotation.
We should mention that the dispersion relation underlying the linear system
\begin{align*}
\begin{cases}
\pa_t u+\Omega e_3\times u+\nabla P=0, &\quad (t,x)\in \R^+\times\R^3,\\
\mathrm{div\,} u=0,&\quad (t,x)\in \R^+\times\R^3,\\
u(t=0,x)=u_0(x), &\quad x\in \R^3,
\end{cases}
\end{align*}
is given by
$
\Lambda(\xi)=\Omega\frac{\xi_3}{|\xi|}
$.
Compared with the classical Euler equations, we would like to emphasize that this dispersive mechanism exhibits $O(t^{-1})$ decay rate in $L^{\infty}$-norm (see \cite{R}) and is strongly anisotropic and degenerate. It is still surprising that, in the absence of viscosity, the Coriolis force alone is sufficient to stabilize the solutions globally in time in the full 3D setting (see \cite{YGUO}). It is found that the strong rotational effect enhances the temporal decay rate of a certain norm of
the velocity (see \cite{AKL}).
\subsection{Motivation}

A fundamental challenge in mathematical physics is to understand the behavior of solutions for the complex rotating fluids. From the PDE's point of view, it is crucial to know if an equation which models a physical phenomenon is well-posed in the Hadamard's sense: existence, uniqueness, and continuous dependence of the solutions with respect to the initial data. In particular, the lack of continuous dependence would cause incorrect solutions or non meaningful solutions. It is known that, the non-uniform continuity of data-to-solution map suggests that the local well-posedness cannot be established by the contraction mappings principle since this would imply Lipschitz continuity for the solution map.  For certain types of evolutionary equations, for example, the Korteweg-de Vries (KdV) and nonlinear Schr\"{o}dinger equations, the flow map is Lipschitz continuous. Paticularly, the KdV equation is shown to be well-posed in $H^s(\R)$ with $s>-3/4$, then the solution map is Lipschitz on the same $H^s(\R)$ (see \cite{Kenig}). However, for the Burgers equation
$\partial_t u+u\pa_xu=0$, which is the most simple quasilinear symmetric hyperbolic equation,
Kato \cite{Kato75} showed that the dependence on initial data of the Burgers equation in $H^k$ with $k \geq 2$ is continuous but not H\"{o}lder continuous with any prescribed exponent. This example indicates that the continuity of the solution map for a symmetric hyperbolic system is a delicate issue. Chen-Liu-Zhang \cite{CLZ} proved that the solution map of the b-family equation is H\"{o}lder continuous as a map from a bounded set of $H^s(\R),s > 3/2$ with $H^r(\R) (0 \leq r < s)$ topology, to $C([0, T ], H^r(\R))$ for some $T > 0$. In this paper, for the Cauchy problem \eqref{CE}, we shall study the continuity of solution with respect to time variable or uniform continuity of solution with respect to the initial data.

Assume that $u_0\in B^s_{p,r}$, based on the local well-posedness theory for the Euler equations (see \cite[Theorem 7.1]{B}), we know that there exists a solution $u\in \mathcal{C}([0,T];B^s_{p,r})$ for the Euler-Coriolis system \eqref{CE}. Naturally, we may wonder whether or not the solution $u$ can belong to $\mathcal{C}^\alpha([0,T];B^s_{p,r})$ with some $\alpha\in(0,1)$. Therefore, we are interested in the following question:
$$u_0\in B^s_{p,r}\;\Rightarrow\; \exists \;u\in \mathcal{C}([0,T];B^s_{p,r})\;\overset{?}{\Rightarrow}\; u\in\mathcal{C}^\alpha([0,T];B^s_{p,r}) \quad\text{with}\; \alpha\in(0,1).$$

If the initial data $u_0$ has more regularity, that is $u_0\in B^{s'}_{p,r}$ for some $s'>s$, by the interpolation argument, we can deduce that $u\in \mathcal{C}^\alpha([0,T];B^s_{p,r})$ with $\alpha=s'-s$.
In this paper, we will show that there exits initial data $u_0\in B^s_{p,r}$ such that the corresponding solution of the Euler-Coriolis system \eqref{CE} can not belong to $\mathcal{C}^\alpha([0,T];B^s_{p,r})$ with any $\alpha\in(0,1)$.

\subsection{Main Result}
We can now state our main result as follows.

\begin{theorem}\label{th1}
Assume that $u_0\in B^s_{p,r}(\R^3)$ with $(s,p,r)$ satisfying
\begin{align}\label{eq:spr1}
s>\frac{3}{p}+1, (p,r)\in (1,\infty)\times [1,\infty] \quad   \mathrm{or}    \quad s=\frac{3}{p}+1, (p,r)\in (1,\infty)\times\{1\}.
\end{align}
Then there exists a short time $T>0$ such that the Euler-Coriolis system \eqref{CE} admits a unique solution $u\in L^\infty([0,T];B^s_{p,r})$ satisfying that
\begin{align}\label{hx}
\|u(t)\|_{L^\infty_TB^s_{p,r}}\les\|u_0\|_{B^s_{p,r}}.
\end{align}
Furthermore, if $r<\infty$, the above solution $u\in C([0,T];B^s_{p,r})$ is continuous dependent in $B^s_{p,r}$.
\end{theorem}

\begin{remark}
Following the proof of the local well-posedness results for the Euler equations (see \cite[Theorem 7.1]{B} and \cite[Theorem 1.1]{GLY}), we can prove Theorem \ref{th1}. Since the procedure is standard, we omit it.
\end{remark}

Next, we consider the property of continuous dependence of solutions of the Cauchy
problem for \eqref{CE} and show that the solution map for the Euler-Coriolis system \eqref{CE} cannot be uniformly continuous in Besov spaces $B^s_{p,r}$. From now on, we denote
 \begin{align*}
&\mathbf{S}_{t}(u_0)=\text{\rm the solution map of}\, \eqref{CE}\, \text{\rm with initial data}\, u_0
 \end{align*}
and denote any bounded subset $U_R$ in $B^s_{p,r}(\mathbb{R}^3)$ by
$$U_R:=\left\{u_0\in B^s_{p,r}(\mathbb{R}^3): \|u_0\|_{B^s_{p,r}(\mathbb{R}^3)}\leq R,\;\mathrm{div} u_0=0\right\}.$$
\begin{theorem}\label{th2}
Assume that $(s,p,r)$ satisfying \eqref{eq:spr1}.  The data-to-solution map $u_0\mapsto \mathbf{S}_{t}(u_0)$ of the Cauchy problem \eqref{CE} is not uniformly continuous from any bounded subset in $U_R\subset B^s_{p,r}$ into $\mathcal{C}([0,T];B^s_{p,r})$. More precisely, there exist two initial data sequences $f_n+g_n$ and $f_n$ such that
\bbal
&\f\|(f_n)^{(2)}\g\|_{B^s_{p,r}}\approx 1 \quad\text{and}\\
 &\lim_{n\rightarrow \infty}\f\|(f_n)^{(1)}\g\|_{B^s_{p,r}}=\lim_{n\rightarrow \infty}\f\|(f_n)^{(3)}\g\|_{B^s_{p,r}}=\lim_{n\rightarrow \infty}\|g_n\|_{B^s_{p,r}}= 0,
\end{align*}
but the distance of second component of the difference between the solution sequences $\mathbf{S}_t(f_n+g_n)$ and $\mathbf{S}_t(f_n)$
\bbal
\liminf_{n\rightarrow \infty}\f\|\f(\mathbf{S}_t(f_n+g_n)-\mathbf{S}_t(f_n)\g)^{(2)}\g\|_{B^s_{p,r}}\gtrsim t,  \quad \forall \;t\in[0,T].
\end{align*}
\end{theorem}

\begin{remark}
We should mention that Theorem \ref{th2} is stronger than that for the Euler equations in \cite{HM1}. On the one hand, the non-uniform continuous is established in the more general framework of Besov spaces. On the other hand, Theorem \ref{th2} is focused on the second component of data-to-solution map which leads to the non-uniform continuous.
\end{remark}

Moreover, we obtain the following

\begin{theorem}\label{th3}
Assume that $(s,p,r)$ satisfying \eqref{eq:spr1}. For any $\alpha\in(0,1)$, there exists $u_0\in B^s_{p,r}(\R^3)$  such that the data-to-solution map $u_0\mapsto \mathbf{S}_{t}(u_0)\in \mathcal{C}([0,T];B^s_{p,r})$ of the Cauchy problem \eqref{CE} satisfies
\bbal
\limsup_{t\to0^+}\frac{\f\|\mathbf{S}_t(u_0)-u_0\g\|_{B^s_{p,r}}}{t^\alpha}=+\infty.
\end{align*}
\end{theorem}

\begin{remark}
We would like to mention that Theorem \ref{th3} is new. In fact, Theorem \ref{th1} tells us that the solution $\mathbf{S}_{t}(u_0)$  for \eqref{CE} is continuous in time
in Besov spaces $B^s_{p,r}$ with $r<\infty$ while Theorem \ref{th3} furthermore indicates that the solution $\mathbf{S}_{t}(u_0)$  for \eqref{CE} cannot be H\"{o}lder continuous in time in the same Besov spaces $B^s_{p,r}$.
\end{remark}

Finally, as a by-product, we have
\begin{theorem}\label{th4}
Assume that $(s,p,r=\infty)$ satisfying \eqref{eq:spr1}. There exists $u_0\in B^s_{p,\infty}(\R^3)$  such that the data-to-solution map $u_0\mapsto \mathbf{S}_{t}(u_0)$ of the Cauchy problem \eqref{CE} satisfies
\bbal
\limsup_{t\to0^+}\f\|\mathbf{S}_t(u_0)-u_0\g\|_{B^s_{p,\infty}}\geq \ep_0>0.
\end{align*}
\end{theorem}

\begin{remark}
Theorem \ref{th4} implies the ill-posedness of \eqref{CE} in the weaker Besov spaces $B^s_{p,\infty}(\R^3)$ in the sense that the solution map to this system starting from $u_0$ is discontinuous at $t = 0$ in the metric of $B^s_{p,\infty}$.
\end{remark}

\begin{remark}
As mentioned above, \eqref{CE} with $\Omega=0$ is the classical Euler equations. Thus our main Theorems \ref{th1}-\ref{th4} hold true for the classical Euler equations.
\end{remark}

\subsection{Strategies to the proof}

We comment on a few points of main ideas and difficulties.

{\bf Proof of Theorem \ref{th1}.} In the case when $p\in(1,\infty)$, since the Coriolis force has no impact on the well-posedness results, we can prove Theorem \ref{th1} by following the proof of the local well-posedness results for the Euler equations (see \cite[Theorem 7.1 ]{B} and \cite[Theorem 1.1 ]{GLY}).

{\bf Proof of Theorem \ref{th2}.} The main task is to construct initial
data such that the second component of the initial velocity is uniformly bounded in $B^s_{p,r}$, and in the same time to have the other component of the initial velocity remain small. We first choose two sequences of initial values $u^n_{0,1}=f_n+g_n$ and $u^n_{0,2}=f_n$, which can generate the approximate solutions $\mathbf{S}_t(u^n_{0,1})$ and $\mathbf{S}_t(u^n_{0,2})$, respectively. Next, we perform perturbation analysis (See Proposition \ref{pro1}) to show the second component of the  difference between the data-to-solution maps $\mathbf{S}_t(u^n_{0,1})$ and $\mathbf{S}_t(u^n_{0,2})$ is bounded below by a positive constant at any later time,  which means the solution map is not uniformly continuous.

{\bf Proof of Theorem \ref{th3}.}  In order to prove the failure of H\"{o}lder regularity of solutions for the \eqref{CE}, we would like to first find initial condition to generate the corresponding solutions to \eqref{CE} satisfying that
\bbal
\frac{\f\|\mathbf{S}_t(u_0)-u_0\g\|_{B^s_{p,r}}}{t^\alpha}\geq C(t)\to+\infty\quad\text{as} \quad t\to 0^+.
\end{align*}
Our key idea is to decompose the  difference between the data-to-solution map $\mathbf{S}_t(u_0)$ and initial data $u_0$ as
\bbal
&\mathbf{S}_{t}(u_0)-u_0=\mathbf{S}_{t}(u_0)-u_0+t\mathbf{v}_0(u_0)-t\mathbf{v}_0(u_0),
\end{align*}
here and in what follows we denote $\mathbf{v}_0(u_0):=\mathcal{P}(u_0\cd\na u_0+\Omega e_3\times u_0)$.

There remains two main difficulties that we need to overcome
\begin{itemize}
\item First, we perform perturbation analysis to prove that $\mathbf{S}_{t}(u_0)-u_0+t\mathbf{v}_0(u_0)$ is small;
\item Second, we need to ensure that the convect term $u_0\cd\na u_0$ is large.
\end{itemize}
To carry out the analysis for these above two points, a big part of the technical difficulties lies in the construction of our initial data. We construct a new initial data which is somewhat different from Theorem \ref{th2}. Finally, thanks to commutator estimates and some basic analysis, we obtain the loss of H\"{o}lder regularity of solutions to \eqref{CE}.

{\bf Proof of Theorem \ref{th4}.} The main point is to modify the initial data constructed in Theorem \ref{th3}. Following the same procedure as that in Theorem \ref{th3}, we can prove Theorem \ref{th4}.

\subsection{Organization of our paper}

The paper is divided as follows.
\begin{itemize}
  \item In Section \ref{sec2}, we list some notations and known results which will be used in the sequel.
  \item In Section \ref{sec3}, we show the non-uniform continuity of the solution map (Theorem \ref{th2}).
  \item In Section \ref{sec4}, we prove the failure of H\"{o}lder regularity of the solution map (Theorem \ref{th3}).
  \item In Section \ref{sec5}, we establish the discontinuity at zero time of the solution map (Theorem \ref{th4}).
\end{itemize}

\section{Preliminaries}\label{sec2}
We will use the following notations throughout this paper.
\begin{itemize}
  \item For $X$ a Banach space and $I\subset\R$, we denote by $\mathcal{C}(I;X)$ the set of continuous functions on $I$ with values in $X$. Sometimes we will denote $L^p(0,T;X)$ by $L_T^pX$.
  \item The symbol $\mathrm{A}\lesssim (\gtrsim)\mathrm{B}$ means that there is a uniform positive ``harmless" constant $\mathrm{C}$ independent of $\mathrm{A}$ and $\mathrm{B}$ such that $\mathrm{A}\leq(\geq) \mathrm{C}\mathrm{B}$, and we sometimes use the notation $\mathrm{A}\approx \mathrm{B}$ means that $\mathrm{A}\lesssim \mathrm{B}$ and $\mathrm{B}\lesssim \mathrm{A}$.
  \item Let us recall that for all $u\in \mathcal{S}'$, the Fourier transform $\mathcal{F}u$, also denoted by $\widehat{u}$, is defined by
$$
\mathcal{F}u(\xi)=\widehat{u}(\xi)=\int_{\R^3}e^{-\mathrm{i}x\cd \xi}u(x)\dd x \quad\text{for any}\; \xi\in\R^3.
$$
  \item The inverse Fourier transform allows us to recover $u$ from $\widehat{u}$:
$$
u(x)=\mathcal{F}^{-1}\widehat{u}(x)=(2\pi)^{-3}\int_{\R^3}e^{\mathrm{i}x\cdot\xi}\widehat{u}(\xi)\dd\xi.
$$
 \item We denote the Leray projection
\bbal
&\mathcal{P}: L^{p}(\mathbb{R}^{3}) \rightarrow L_{\sigma}^{p}(\mathbb{R}^{3}) \equiv \overline{\left\{f \in \mathcal{C}^\infty_{0}(\mathbb{R}^{d}) ; {\rm{div}} f=0\right\}}^{\|\cdot\|_{L^{p}(\mathbb{R}^{3})}},\quad p\in(1,\infty),\\
&\mathcal{Q}=\mathrm{Id}-\mathcal{P}.
\end{align*}
In $\mathbb{R}^{d}$, $\mathcal{P}$ can be defined by $\mathcal{P}= \mathrm{Id}+(-\Delta)^{-1}\nabla {\rm{div}}$, or equivalently, $\mathcal{P}=(\mathcal{P}_{i j})_{1 \leqslant i, j \leqslant 3}$, where $\mathcal{P}_{i j} \equiv \delta_{i j}+R_{i} R_{j}$ with $\delta_{i j}$ being the Kronecker delta ($\delta_{i j}=0$ for $i\neq j$ and $\delta_{i i}=0$) and $R_{i}$ being the Riesz transform with symbol $-\mathrm{i}\xi_i/|\xi|$. Obviously, $\mathcal{Q}= -(-\Delta)^{-1}\nabla {\rm{div}}$, and if $\div\, u=\div\, v=0$, it holds that
$
\mathcal{Q}(u\cdot\na v)= \mathcal{Q}(v\cdot\na u).
$
\end{itemize}
Next, we will recall some facts about the Littlewood-Paley decomposition and the nonhomogeneous Besov spaces (see \cite{B} for more details).
Choose a radial, non-negative, smooth function $\vartheta:\R^3\mapsto [0,1]$ such that
 ${\rm{supp}} \;\vartheta\subset B(0, 4/3)$ and $\vartheta(\xi)\equiv1$ for $|\xi|\leq3/4$.
Setting $\varphi(\xi):=\vartheta(\xi/2)-\vartheta(\xi)$, then we deduce that $\varphi$ has the following properties
\begin{itemize}
  \item ${\rm{supp}} \;\varphi\subset \left\{\xi\in \R^3: 3/4\leq|\xi|\leq8/3\right\}$;
  \item $\varphi(\xi)\equiv 1$ for $4/3\leq |\xi|\leq 3/2$;
  \item $\vartheta(\xi)+\sum_{j\geq0}\varphi(2^{-j}\xi)=1$ for any $\xi\in \R^3$;
  \item $\sum_{j\in \mathbb{Z}}\varphi(2^{-j}\xi)=1$ for any $\xi\in \R^3\setminus\{0\}$.
\end{itemize}
The nonhomogeneous and homogeneous dyadic blocks are defined as follows
\begin{align*}
\forall\, u\in \mathcal{S'}(\R^3),\quad \Delta_ju=0,\; \text{if}\; j\leq-2;\quad
\Delta_{-1}u=\vartheta(D)u;\quad
\Delta_ju=\varphi(2^{-j}D)u,\; \; \text{if}\;j\geq0,
\end{align*}
and
\begin{align*}
\forall\, u\in \mathcal{S}'_h(\R^3),\quad
\dot{\Delta}_ju=\varphi(2^{-j}D)u,\; \; \text{if}\;j\in \mathbb{Z},
\end{align*}
where the pseudo-differential operator is defined by $\sigma(D):u\to\mathcal{F}^{-1}(\sigma \mathcal{F}u)$ and $\mathcal{S}'_h$ is given by
\begin{eqnarray*}
\mathcal{S}'_h:=\Big\{u \in \mathcal{S'}(\mathbb{R}^{3}):\; \lim_{j\rightarrow-\infty}\|\vartheta(2^{-j}D)u\|_{L^{\infty}}=0 \Big\}.
\end{eqnarray*}
We recall the definition of the Besov spaces and norms.
\begin{definition}[Nonhomogeneous Besov spaces, see \cite{B}]
Let $s\in\mathbb{R}$ and $(p,r)\in[1, \infty]^2$. We define the nonhomogeneous Besov spaces
$$
B^{s}_{p,r}:=\f\{f\in \mathcal{S}':\;\|f\|_{B^{s}_{p,r}}:=\left\|2^{js}\|\Delta_jf\|_{L_x^p}\right\|_{\ell^r(j\geq-1)}<\infty\g\}.
$$
\end{definition}
\begin{definition}[Homogeneous Besov spaces, see \cite{B}]
Let $s\in\mathbb{R}$ and $(p,r)\in[1, \infty]^2$. We define the homogeneous Besov spaces
$$
\dot{B}^{s}_{p,r}:=\f\{f\in \mathcal{S}'_h:\;\|f\|_{\dot{B}^{s}_{p,r}}:=\left\|2^{js}\|\dot{\Delta}_jf\|_{L_x^p}\right\|_{\ell^r(j\in \mathbb{Z})}<\infty\g\}.
$$
\end{definition}
We remark that, for any $s>0$ and $(p,r)\in[1, \infty]^2$, then $B^{s}_{p,r}(\R^3)=\dot{B}^{s}_{p,r}(\R^3)\cap L^p(\R^3)$ and
$$\|f\|_{B^{s}_{p,r}}\approx \|f\|_{L^{p}}+\|f\|_{\dot{B}^{s}_{p,r}}.$$

The following Bernstein's inequalities will be used in the sequel.
\begin{lemma}[\cite{B}] \label{lem2.1} Let $\mathcal{B}$ be a ball and $\mathcal{C}$ be an annulus. There exists a constant $C>0$ such that for all $k\in \mathbb{N}\cup \{0\}$, any $\lambda\in \R^+$ and any function $f\in L^p$ with $1\leq p \leq q \leq \infty$, we have
\begin{align*}
&{\rm{supp}}\ \widehat{f}\subset \lambda \mathcal{B}\;\Rightarrow\; \|D^kf\|_{L^q}\leq C^{k+1}\lambda^{k+(\frac{1}{p}-\frac{1}{q})}\|f\|_{L^p},  \\
&{\rm{supp}}\ \widehat{f}\subset \lambda \mathcal{C}\;\Rightarrow\; C^{-k-1}\lambda^k\|f\|_{L^p} \leq \|D^kf\|_{L^p} \leq C^{k+1}\lambda^k\|f\|_{L^p}.
\end{align*}
\end{lemma}
Next we recall the following product law which will be used often in the sequel.
\begin{lemma}[\cite{B}]\label{lp}
Assume $(s,p,r)$ satisfies \eqref{eq:spr1} and $\sigma>0$. Then
 there exists a constant $C$, depending only on $d,p,r,\sigma$ or $s$ such that
$$\|fg\|_{B^{\sigma}_{p,r}}\leq C\f(\|f\|_{L^\infty}\|g\|_{B^\sigma_{p,r}}+\|g\|_{L^\infty}\|f\|_{B^{\sigma}_{p,r}}\g),\quad \forall f,g\in L^\infty\cap B^{\sigma}_{p,r}.$$
Furthermore, due to the embedding $B^{s-1}_{p,r}(\R^3)\hookrightarrow L^{\infty}(\R^3)$, there holds
$$\|f\cdot \na g\|_{B^{s-1}_{p,r}}\leq C\|f\|_{B^{s-1}_{p,r}}\|g\|_{B^s_{p,r}},\quad \forall(f,g)\in B^{s-1}_{p,r}\times B^s_{p,r}$$
and
\begin{align}\label{cj}
\|f\cdot \na g\|_{B^{s}_{p,r}}\leq C\f(\|f\|_{B^{s-1}_{p,r}}\|g\|_{B^{s+1}_{p,r}}+\|f\|_{B^{s}_{p,r}}\|g\|_{B^{s}_{p,r}}\g),\quad \forall(f,g)\in B^{s}_{p,r}\times B^{s+1}_{p,r}.
\end{align}
\end{lemma}

\begin{lemma}[\cite{GLY}]\label{lem:P}
Assume that $(s,p,r)$ satisfies \eqref{eq:spr1}. Then there exists a constant $C$, depending only on $d,p,r,s$, such that for all $u,v\in B^s_{p,r}$ with $\mathrm{div\,} u=\mathrm{div\,} v=0$
\begin{align*}
&\|\mathcal{Q}(u\cdot \na v)\|_{B^s_{p,r}}\leq C \f(\|u\|_{C^{0,1}}\|v\|_{B^s_{p,r}}+\|v\|_{C^{0,1}}\|u\|_{B^s_{p,r}}\g)
\end{align*}
with $\mathcal{Q}= -(-\Delta)^{-1}\nabla {\rm{div}}$. Furthermore, there holds
\begin{align*}
&\|\mathcal{Q}(u\cdot \na v)\|_{B^s_{p,r}}\leq C \|u\|_{B^s_{p,r}}\|v\|_{B^s_{p,r}},\\
&\|\mathcal{Q}(u\cdot \na v)\|_{B^{s-1}_{p,r}}\leq C \min\f\{\|u\|_{B^{s-1}_{p,r}}\|v\|_{B^s_{p,r}},\, \|v\|_{B^{s-1}_{p,r}}\|u\|_{B^s_{p,r}}\g\}.
\end{align*}
\end{lemma}

\section{Proof of Theorem \ref{th2}}\label{sec3}

In this section, we prove Theorem \ref{th2}.
Letting $\widehat{\phi}\in \mathcal{C}^\infty_0(\mathbb{R})$ be an even, real-valued and non-negative function on $\R$ and satisfy
\begin{align*}
{\widehat{\phi}(\xi)=}\begin{cases}
1,&\text{if}\quad |\xi|\leq \frac{1}{64},\\
0,&\text{if}\quad |\xi|\geq \frac{1}{8}.
\end{cases}
\end{align*}
The inverse Fourier transform allows us to recover $\phi$ from $\widehat{\phi}$, namely, $\phi(x)=\mathcal{F}^{-1}(\widehat{\phi}(\xi))$. Then we can define the new real-valued function
$$\theta(x)=\phi(x_1)\phi(x_2)\phi(x_3).$$
It is easy to check that $\theta(0)=\phi^3(0)>0$ and
\bal\label{s1}
\mathrm{supp} \ \widehat{\theta}\subset \left\{\xi=(\xi_1,\xi_2,\xi_3)\in\R^3: \ |\xi|\leq  \frac{\sqrt{3}}{8}\right\}.
\end{align}
Thus for $\lambda\gg1$ and $T\in\{\sin,\,\cos\}$
\bal\label{s2}
\mathrm{supp} \ \mathcal{F}\Big(\theta(x)T(\lambda x_1)\Big)\subset \left\{\xi\in\R^3: \ \lambda-\fr18\leq|\xi_1|\leq \lambda+\frac{1}{8},\ |\xi_2|\leq \frac{1}{8},\ |\xi_3|\leq\frac{1}{8} \right\}.
\end{align}
For any $p\in[1,\infty]$, there exist two positive constants $C_1$ and $C_2$ such that
\begin{align*}
C_1\leq\|\phi\|_{L^p(\R)}\leq C_2.
\end{align*}
{\bf Definition 1.}\, Let $(s,p,r)$ satisfies \eqref{eq:spr1} and $n\gg1$. We define the high frequency function $f_n$ and the low frequency function $g_n$ by
\bbal
&f_n(x)=2^{-n(s+1)}
\f(-\pa_2,\pa_1,0\g)
\f[\theta(x)\cos \f(\frac{17}{12}2^nx_1\g)\g],\\
&g_n(x)=2^{-n}
\f(-\pa_2,\pa_1,0\g)\theta(x).
\end{align*}
\begin{remark}\label{100} The choice of $\frac{17}{12}\in(\frac{4}{3},\frac{3}{2})$ is crucial but not unique. In fact,
due to \eqref{s2}, which implies that
\bbal
\mathrm{supp}\ \mathcal{F}\f(f_n\g)\subset \left\{\xi\in\R^3: \ \frac{17}{12}2^n-1\leq |\xi|\leq \frac{17}{12}2^n+1\right\},
\end{align*}
and notice that
$
\varphi(2^{-j}\xi)\equiv 1$ in $\left\{\xi\in\R^3: \ \frac{4}{3}2^{j}\leq |\xi|\leq \frac{3}{2}2^{j}\right\},
$
we have
$
\mathcal{F}(\dot{\Delta}_jf_n)=\varphi(2^{-j}\cdot)\widehat{f_n}=0$ for  $j\neq n,
$
and thus
$
\dot{\Delta}_jf_n=
f_n$ if $j=n$.
\end{remark}
\begin{lemma}\label{y1}
Let $f_n$ and $g_n$ be defined as above. There exists a positive constant $C$ such that for any $\sigma\in\R$
\bbal
&\|f_n^{(1)}\|_{B^\sigma_{p,r}}\leq C 2^{n(\sigma-s-1)},\\
&\|f_n^{(2)}\|_{B^\sigma_{p,r}}\thickapprox 2^{n(\sigma-s)},\\
&\|f_{n}\|_{B_{p,r}^{\sigma}}\thickapprox 2^{n(\sigma-s)},\\
&\|g_{n}\|_{B_{p,r}^{\sigma}}\thickapprox 2^{-n}.
\end{align*}
\end{lemma}
\begin{proof}\,
From Remark \ref{100} and Bernstein's inequality (see Lemma \ref{lem2.1}),  we obtain the desired results of Lemma \ref{y1}.
\end{proof}
\begin{lemma}\label{ZZ}
Let $f_n$ and $g_n$ be defined as above. There exists a positive constant $c$ such that
\bal\label{yZ}
&\liminf_{n\rightarrow \infty}\f\|g_n\cd\na f_n^{(2)}\g\|_{B^s_{p,\infty}}\geq c.
\end{align}
\end{lemma}
\begin{proof}  \, Due to \eqref{s1} and \eqref{s2}, we have
\bbal
\mathrm{supp}\ \mathcal{F}\f(g_n\cd \na f_n^{(2)}\g)\subset \left\{\xi\in\R^3: \ \frac{17}{12}2^n-1\leq |\xi|\leq \frac{17}{12}2^n+1\right\},
\end{align*}
which implies
\begin{align*}
{\Delta_j\f(g_n\cd\na f_n^{(2)}\g)=}
\begin{cases}
g_n\cd \na f_n^{(2)}, &\text{if}\; j=n,\nonumber\\
0, &\text{if}\; j\neq n.\nonumber
\end{cases}
\end{align*}
Thus, we have
\bal\label{ml}
\f\|g_n\cd \na f_n^{(2)}\g\|_{B^s_{p,\infty}}&=2^{ns}\f\|g_n\cd \na f_n^{(2)}\g\|_{L^p}.
\end{align}
The definitions of $f_n$ and $g_n$ directly tell us that
\bbal
2^{ns}g_n\cd \na f_n^{(2)}&=-\left(\frac{17}{12}\right)^2\phi^2(x_1)\cos \left(\frac{17}{12}2^nx_1\right)\phi(x_2)\phi'(x_2)\phi^2(x_3)+\mathbf{R}.
\end{align*}
For the remainder terms $\mathbf{R}$, noticing that at most one partial derivative falls on cosine function, then $\|\mathbf{R}\|_{L^p}\leq C2^{-n}$.
Thus, we have
\bbal
&2^{ns}\|g_n\cd \na f_n^{(2)}\|_{L^p}\geq \left\|\phi^2(x)\cos \left(\frac{17}{12}2^nx\right)\right\|_{L^p(\R)}\|\phi\|^{2}_{L^{2p}(\R)}\f\|\phi(x)\phi'(x)\g\|_{L^p(\R)}-C2^{-n},
\end{align*}
Plugging the above into \eqref{ml} yields
\bal\label{m9}
\|g_n\cd \na f_n\|_{B^s_{p,\infty}}&\geq c\left\|\phi^2(x)\cos \left(\frac{17}{12}2^nx\right)\right\|_{L^p(\R)}-C2^{-n}.
\end{align}
Notice that the fact
\begin{align*}
\liminf_{n\rightarrow \infty}\left\|\phi^2(x)\cos \left(\fr{17}{12}2^nx\right)\right\|_{L^p(\R)}\geq c>0,
\end{align*}
from \eqref{m9}, we get the desired result \eqref{yZ}. Thus we finish the proof of Lemma \ref{ZZ}.\end{proof}

\begin{proposition}\label{pro1}
Under the assumptions of Theorem \ref{th2}, we have
\bal\label{et0}
\|\mathbf{S}_{t}(u_0)-u_0+t\mathbf{v}_0(u_0)\|_{B^{s}_{p,r}}\leq Ct^{2}\mathbf{E}(u_0),
\end{align}
here and in what follows we denote
\bbal
\mathbf{E}(u_0)&:=1+\|u_0\|_{B^{s-1}_{p,r}}\f(\|u_0\|_{B^{s+1}_{p,r}}+
\|u_0\|_{B^{s-1}_{p,r}}
\|u_0\|_{B^{s+2}_{p,r}}\g).
\end{align*}
\end{proposition}
\begin{proof}\,
For simplicity, we denote $u(t)=\mathbf{S}_t(u_0)$ here and in what follows. By Theorem \ref{th1}, one has
\bal\label{bound:u}
\|u(t)\|_{B^{s+k}_{p,r}}\leq C\|u_0\|_{B^{s+k}_{p,r}}\quad \text{for }\; k\in\{0,\pm1,2\}.
\end{align}
Using the Mean Value Theorem and Lemma \ref{lp}, we obtain from \eqref{CE} that
\bal\label{et1}
&\|u(t)-u_0\|_{B^{s-1}_{p,r}}\leq\int^t_0\|\pa_\tau u\|_{B^{s-1}_{p,r}} \dd\tau
\nonumber\\
\leq&~ \int^t_0\f(\|u\cd\na u\|_{B^{s-1}_{p,r}}+\|\mathcal{P}(\Omega e_3\times u)\|_{B^{s-1}_{p,r}}+\|\mathcal{Q}(u\cd \na u)\|_{B^{s-1}_{p,r}}\g)\dd \tau
\nonumber\\
\leq&~  C\int^t_0\f(\|u\|_{B^s_{p,r}}\|u\|_{B^{s-1}_{p,r}}+\|u\|_{B^{s-1}_{p,r}}\g)\dd \tau\nonumber\\
\leq&~  Ct\|u_0\|_{B^{s-1}_{p,r}},
\end{align}
\bal\label{et2}
&\|u(t)-u_0\|_{B^{s}_{p,r}}\leq\int^t_0\|\pa_\tau u\|_{B^{s}_{p,r}} \dd\tau\nonumber\\
\leq&~ \int^t_0\f(\|u\cd\na u\|_{B^{s}_{p,r}}+\|\mathcal{P}(\Omega e_3\times u)\|_{B^{s}_{p,r}}+\|\mathcal{Q}(u\cd \na u)\|_{B^{s}_{p,r}}\g)\dd \tau\nonumber\\
\leq&~ Ct\f(1+\|u_0\|_{B^{s-1}_{p,r}}\|u_0\|_{B^{s+1}_{p,r}}\g),
\end{align}
\bal\label{et3}
&\|u(t)-u_0\|_{B^{s+1}_{p,r}}\leq\int^t_0\|\pa_\tau u\|_{B^{s+1}_{p,r}} \dd\tau
\nonumber\\
\leq&~ \int^t_0\f(\|u\cd\na u\|_{B^{s+1}_{p,r}}+\|\mathcal{P}(\Omega e_3\times u)\|_{B^{s+1}_{p,r}}+\|\mathcal{Q}(u\cd \na u)\|_{B^{s+1}_{p,r}}\g)\dd \tau
\nonumber\\
\leq&~ Ct\f(\|u_0\|_{B^{s+1}_{p,r}}+\|u_0\|_{B^{s-1}_{p,r}}\|u_0\|_{B^{s+2}_{p,r}}\g).
\end{align}
Using the Mean Value Theorem once again, we obtain that
\bal\label{et4}
&\|u(t)-u_0+t\mathbf{v}_0(u_0)\|_{B^{s}_{p,r}}\leq \int^t_0\|\pa_\tau u+\mathbf{v}_0(u_0)\|_{B^{s}_{p,r}} \dd\tau
\nonumber\\
\leq&~\int^t_0\|\mathcal{Q}(u)-\mathcal{Q}(u_0)\|_{B^{s}_{p,r}} \dd\tau+\int^t_0\|u\cd\na u-u_0\cd\na u_0\|_{B^{s}_{p,r}} \dd\tau\nonumber\\
\quad&+\Omega\int^t_0\|\mathcal{P}\f(e_3\times (u-u_0)\g)\|_{B^{s}_{p,r}}\dd \tau
\nonumber\\
\lesssim&~ \int^t_0\|u(\tau)-u_0\|_{B^{s}_{p,r}} \dd\tau+\int^t_0\|u(\tau)-u_0\|_{L^\infty} \|u(\tau)\|_{B^{s+1}_{p,r}} \dd\tau
\nonumber\\
\quad & + \int^t_0\|u(\tau)-u_0\|_{B^{s+1}_{p,r}}  \|u_0\|_{L^\infty}\dd \tau\nonumber\\
\lesssim&~ \int^t_0\|u(\tau)-u_0\|_{B^{s}_{p,r}} \dd\tau+\|u_0\|_{B^{s+1}_{p,r}}\int^t_0\|u(\tau)-u_0\|_{L^\infty}  \dd\tau\nonumber\\
\quad &+ \|u_0\|_{L^\infty}\int^t_0\|u(\tau)-u_0\|_{B^{s+1}_{p,r}}  \dd \tau.
\end{align}
Plugging \eqref{et1}-\eqref{et3} into \eqref{et4} and using the fact that $B^{s-1}_{p,r}(\R^3)\hookrightarrow L^\infty(\R^3)$, yields the desired result of Proposition \ref{pro1}.
\end{proof}

{\bf Proof of Theorem \ref{th2}.} Let $u^n_0:=f_n+g_n$. Obviously, we have
\bbal
\|u^n_0-f_n\|_{B^s_{p,r}}=\|g_n\|_{B^s_{p,r}}\leq C2^{-n},
\end{align*}
which means that
\bbal
\lim_{n\to\infty}\|u^n_0-f_n\|_{B^s_{p,r}}=0.
\end{align*}
Obviously, we obtain from Lemma \ref{y1} that
\bbal
\|u^n_0\|_{B^{s+k}_{p,r}}\leq C2^{kn} \quad \text{for }\; k\in\{0,\pm1,2\}.
\end{align*}
We decompose the solution maps $\mathbf{S}_{t}(u^n_0)$ and $\mathbf{S}_{t}(f_n)$ as follows
\bbal
\mathbf{S}_{t}(u^n_0)&=\mathbf{S}_{t}(u^n_0)-u^n_0+t\mathbf{v}_0(u^n_0)+(f_n+g_n)-t(f_n+g_n)\cd\na(f_n+g_n)\\
&\quad+t\mathcal{Q}\big[(f_n+g_n)\cd\na(f_n+g_n)\big]-t\mathcal{P}(\Omega e_3\times (f_n+g_n)),\\
\mathbf{S}_{t}(f_n)&=\mathbf{S}_{t}(f_n)-f_n+t\mathbf{v}_0(f_n)+f_n-tf_n\cd\na f_n\\
&\quad+t\mathcal{Q}(f_n\cd\na f_n)-t\mathcal{P}(\Omega e_3\times f_n).
\end{align*}
Then, by Proposition \ref{pro1}, we deduce that
\bal\label{zz}
&\f\|\mathbf{S}^{(2)}_{t}(u^n_0)-\mathbf{S}^{(2)}_{t}(f_n)\g\|_{B^s_{p,r}}\nonumber\\
\geq&~ t\f\|\f[(f_n+g_n)\cd\na(f_n+g_n)-f_n\cd\na f_n\g]^{(2)}\g\|_{B^s_{p,r}}-t\|\mathcal{P}(\Omega e_3\times g_n)\|_{B^s_{p,r}} \nonumber\\
\quad&-t\f\|\mathcal{Q}\big[(f_n+g_n)\cd\na(f_n+g_n)-f_n\cd\na f_n\big]\g\|_{B^s_{p,r}}-\|g_n\|_{B^s_{p,r}}-Ct^2
\nonumber\\
 \geq&~ t\f\|g_n\cd\na f_n^{(2)}\g\|_{B^s_{p,r}}-Ct^2-C2^{-n},
\end{align}
where we have used the facts from Lemma \ref{lem:P} and Lemma \ref{y1}
\bbal
&\|\mathcal{Q}(f_n\cd\na g_n)\|_{B^s_{p,r}}\les \|f_n\cd\na g_n\|_{B^s_{p,r}}\les \|f_n\|_{B^s_{p,r}}\|g_n\|_{B^{s+1}_{p,r}}\les2^{-n},\\
&\|\mathcal{Q}(g_n\cd\na g_n)\|_{B^s_{p,r}}\les\|g_n\cd\na g_n\|_{B^s_{p,r}}\les \|g_n\|_{B^s_{p,r}}\|g_n\|_{B^{s+1}_{p,r}}\les2^{-2n},\\
&\|\mathcal{Q}(g_n\cd\na f_n)\|_{B^s_{p,r}}=\|\mathcal{Q}(f_n\cd\na g_n)\|_{B^s_{p,r}}\les \|f_n\|_{B^s_{p,r}}\|g_n\|_{B^{s+1}_{p,r}}\les2^{-n}.
\end{align*}
Using Lemma \ref{ZZ}, we obtain from \eqref{zz} that
\bbal
\f\|\mathbf{S}^{(2)}_{t}(u^n_0)-\mathbf{S}^{(2)}_{t}(f_n)\g\|_{B^s_{p,r}}&\geq ct-Ct^2-C2^{-n},
\end{align*}
which yields the result of Theorem \ref{th2}. Thus we complete the proof of Theorem \ref{th2}.

\section{Proof of Theorem \ref{th3}}\label{sec4}

Assume that $(s,p,r)$ satisfies \eqref{eq:spr1}. We define the initial data $u_0(x)$ as
\bbal
&u_0(x):=\f(-\pa_2,\pa_1,0\g)\sum\limits^{\infty}_{n=3}a_n(x),  \quad\text{where}\\
& a_n(x):= n^{-2}2^{-n(s+1)}\theta(x)\cos \f(\frac{17}{12}2^{n}x_1\g).
\end{align*}

\begin{lemma}\label{ley2} Assume that $s\in \R$ and $(p,r)\in[1,\infty]^2$. Define the divergence-free vector field $u_0$ as above.
Then there exists some sufficiently large $n\in \mathbb{N}$ and some positive constant $c_0$ such that
\begin{align}
&\|u_0\|_{B^s_{p,r}}\thickapprox\|u_0\|_{\dot{B}^s_{p,r}}\thickapprox 1, \label{z1}\\
&2^{ns}\f\|u_0\cd\na \De_{n}u_0^{(2)}\g\|_{L^p}\geq c_0n^{-2}2^{n}.\label{z2}
\end{align}
\end{lemma}
\begin{proof}\,
Due to the simple fact \eqref{s2}, it holds that for some sufficiently large $n\in \mathbb{N}$
 \begin{align*}
 \Delta_{n}u_0^{(1)}=-\pa_2a_n\quad\text{and}\quad\Delta_{n}u_0^{(2)}=\pa_1a_n,
    \end{align*}
 which gives directly that
    \begin{align}\label{u0}
u_0\cd\na \De_{n}u_0^{(2)}&=
-u_0^{(1)}{\pa_1\pa_1}a_n+u_0^{(2)}\pa_1\pa_2a_n.
    \end{align}
\eqref{z1} is obvious.
To estimate \eqref{z2}, we should be emphasize that the leading term is $u_0^{(1)}{\pa_1\pa_1}a_n$ since the other terms in the right hand side of \eqref{u0} can be absorbed by the term $u_0^{(1)}{\pa_1\pa_1}a_n$. Thus we focus on the estimation of the leading term.
\begin{align*}
-n^{2}2^{n(s+1)}u_0^{(1)}{\pa_1\pa_1}a_n&=\f(\fr{17}{12}2^{n}\g)^2\sum_{k=3}^{\infty}k^{-2}2^{-k(s+1)}\phi^2(x_1)\cos \left(\frac{17}{12}2^kx_1\right)\\
&\quad\times \cos \left(\frac{17}{12}2^nx_1\right)\phi^2(x_3)\phi(x_2)\phi'(x_2)
+\text{Remainder terms}.
\end{align*}
For the Remainder terms, noticing that at most one partial derivative falls on cosine function, thus we easily find that
$$\f\|\text{Remainder terms}\g\|_{L^p(\R^3)}\leq C2^{n}.$$
Since $\sum_{k=3}^{\infty}k^{-2}2^{-k(s+1)}\phi^2(x_1)\cos \left(\frac{17}{12}2^kx_1\right)$ is a real-valued and continuous function on $\R$, then there exists some $\delta>0$ such that for any $x\in B_{\delta}(0):=\{x\in\mathbb{R}:\;|x|\leq\delta\}$
\begin{align*}
&\f|\sum_{k=3}^{\infty}k^{-2}2^{-k(s+1)}\phi^2(x_1)\cos \left(\frac{17}{12}2^kx_1\right)\g|
\geq\fr{\phi^2(0)}{2}\sum\limits^{\infty}_{k=3}k^{-2}2^{-k(s+1)}=:c_0>0.
\end{align*}
Thus we have for some sufficiently large $n\in \mathbb{N}$
\begin{align}\label{n1}
n^{2}2^{n(s+1)}\f\|u_0^{(1)}\pa_1\pa_1a_n\g\|_{L^p(\R^3)}&\geq c_02^{2n}-C2^n.
   \end{align}
By easy computations, we have
\begin{align}\label{n2}
n^{2}2^{n(s+1)}\f\|u_0^{(2)}\pa_1\pa_2a_n\g\|_{L^p(\R^3)}&\leq C2^{n}.
   \end{align}
Noticing that \eqref{u0}, from \eqref{n1} and \eqref{n2}, we obtain that for some sufficiently large $n\in \mathbb{N}$
\begin{align*}
n^{2}2^{n(s+1)}\f\|u_0\cd\na \De_{k}u_0^{(2)}\g\|_{L^p(\R^3)}\geq c_02^{2n}-C2^n,
   \end{align*}
which is nothing but the desired result \eqref{z2}. We complete the proof of Lemma \ref{ley2}.
\end{proof}

\begin{proposition}\label{pro2}
Assume that $\|u_0\|_{B^s_{p,r}}\approx1$. Under the assumptions of Theorem \ref{th3}, we have
\begin{align*}
\|\mathbf{S}_{t}(u_0)-u_0+t\mathbf{v}_0(u_0)\|_{{B}^{s-2}_{p,r}}\lesssim t^{2}.
\end{align*}
\end{proposition}
\begin{proof}
It holds that
\bbal
\|u(t)\|_{L^\infty_T(B^s_{p,r})}\les\|u_0\|_{B^s_{p,r}}\les1.
\end{align*}
From \eqref{et1}, we obtain
\bal
\|u(t)-u_0\|_{B^{s-1}_{p,r}}\les t\|u_0\|_{B^{s}_{p,r}}\les t.
\end{align}
By the Newton-Leibniz formula again, we get
\begin{align}
&\|\mathbf{S}_{t}(u_0)-u_0+t\mathbf{v}_0(u_0)\|_{{B}^{s-2}_{p,r}}
\leq \int^t_0\|\partial_\tau u+\mathbf{v}_0\|_{{B}^{s-2}_{p,r}} \dd\tau \nonumber\\
\les&~ \int^t_0\|\mathcal{P}\f(u\cd\na u-u_0\cd\na u_0+\Omega e_3\times (u-u_0)\g)\|_{{B}^{s-2}_{p,r}} \dd\tau
\nonumber\\
\les&~  \int^t_0\f(\|(u-u_0)(u+u_0)\|_{B^{s-1}_{p,r}}+\|u-u_0\|_{B^{s-1}_{p,r}}\g)\dd\tau\nonumber
\\
\les&~ \int^t_0\|u-u_0\|_{B^{s-1}_{p,r}} \dd\tau\nonumber
\les t^2.
\end{align}
We complete the proof of Proposition \ref{pro2}.
\end{proof}

{\bf Proof of Theorem \ref{th3}.}   Setting $\mathbf{w}:=\mathbf{S}_{t}(u_0)-u_0+t\mathbf{v}_0(u_0)$
$$\mathbf{S}_{t}(u_0)-u_0=-t\mathbf{v}_0(u_0)+\mathbf{w} \quad \text{and}\quad \mathbf{v}_0(u_0)=u_0\cd\na u_0-\mathcal{Q}(u_0\cd\na u_0)+\Omega\mathcal{P}(e_3\times u_0).$$

By the triangle inequality, we deduce that
\bal\label{zn}
\f\|\f(\mathbf{S}_{t}(u_0)-u_0\g)^{(2)}\g\|_{B^{s}_{p,r}}
&\geq2^{{ns}}
\f\|\De_{n}\f(\mathbf{S}_{t}(u_0)-u_0\g)^{(2)}\g\|_{L^p}
=2^{{ns}}\f\|\De_{n}\big(t\mathbf{v}_0(u_0)-\mathbf{w}\big)^{(2)}\g\|_{L^p}
\nonumber\\&\geq t2^{{ns}}\f\|\De_{n}\f(\mathbf{v}_0(u_0)\g)^{(2)}\g\|_{L^p}
-2^{{2n}}2^{{n(s-2)}}
\f\|\De_{n}\mathbf{w}\g\|_{L^p}\nonumber\\
&\geq t2^{{n}s}\f\|\De_{n}\big(u_0\cd\na u_0^{(2)}\big)\g\|_{L^p}-
t2^{{n}s}\f\|\De_{n}\mathcal{Q}(u_0\cd\na u_0)\g\|_{L^p}\nonumber\\
&\quad- t2^{{n}s}\f\|\De_{n}\big(\Omega\mathcal{P}(e_3\times u_0))\big)\g\|_{L^p}-C2^{2{n}}\|\mathbf{w}\|_{\dot{B}^{s-2}_{p,\infty}}
\nonumber\\&\geq t2^{{n}s}\f\|u_0\cd\na\De_{n}u_0^{(2)}\g\|_{L^p}-t2^{ns}\f\|[\De_{n},u_0]\cd\na u_0^{(2)}\g\|_{L^p}\nonumber\\
&\quad-
Ct2^{{n}s}\f\|\De_{n}\mathcal{Q}(u_0\cd\na u_0)\g\|_{L^p}-C t\|u_0\|_{B^{s}_{p,\infty}}-C2^{2n}t^2\nonumber\\
&\geq t2^{ns}\f\|u_0\cd\na \De_{n}u_0^{(2)}\g\|_{L^p}-C t-C2^{2{n}}t^2,
\end{align}
where we have used the fact:
\bbal
\big\|2^{{n}s}\|[\De_{n},u_0]\cd\na u_0^{(2)}\|_{L^p}\big\|_{\ell^\infty}\les \|\na u_0\|_{L^\infty}\|u_0\|_{B^{s}_{p,r}}\les 1,
\end{align*}
and
\bbal
2^{{ns}}\|\De_{n}\mathcal{Q}(u_0\cd\na u_0)\|_{L^p}&= 2^{{ns}}\|\De_{n}((-\Delta)^{-1}\nabla(\pa_iu^j_0\pa_j u_0^i))\|_{L^p}\\
&\approx 2^{{n(s-1)}}\|\De_{n}(\pa_iu^j_0\pa_j u_0^i)\|_{L^p}\\
&\les \|\pa_iu^j_0\pa_j u_0^i\|_{B^{s-1}_{p,\infty}}\les \|u_0\|^2_{B^{s}_{p,\infty}}\les1.
\end{align*}
Then, using Lemma \ref{ley2}, we obtain from \eqref{zn} that
\bbal
\f\|\f(\mathbf{S}_{t}(u_0)-u_0\g)^{(2)}\g\|_{B^{s}_{p,r}}\geq ctn^{-2}2^{n}-Ct-C2^{2{n}}t^2,
\end{align*}
which implies
\bbal
t^{-\alpha}\f\|\f(\mathbf{S}_{t}(u_0)-u_0\g)^{(2)}\g\|_{B^{s}_{p,r}}\geq ct^{1-\alpha}n^{-2}2^{n}-Ct^{1-\alpha}-C2^{2{n}}t^{2-\alpha},
\end{align*}
Thus, picking $t^{1-\alpha}_n=n^32^{-n}$ with large $n$, we have
\bbal
t^{-\alpha}_n\f\|\f(\mathbf{S}_{t}(u_0)-u_0\g)^{(2)}\g\|_{B^{s}_{p,r}}&\geq cn-Cn^32^{-n}-Cn^{6}t_n^{\alpha}
\geq \tilde{c}n.
\end{align*}
This completes the proof of Theorem \ref{th3}.
\section{Proof of Theorem \ref{th4}}\label{sec5}

We define the initial data $u_0(x)$ as
\bbal
&u_0(x):=\f(-\pa_2,\pa_1,0\g)\sum\limits^{\infty}_{n=3}b_n(x), \quad\text{where}\\
&b_n(x):= 2^{-n(s+1)}\theta(x)\cos \f(\frac{17}{12}2^{n}x_1\g).
\end{align*}
Following similar procedure as that in Theorem \ref{th3}, then we have
\bbal
\f\|\f(\mathbf{S}_{t}(u_0)-u_0\g)^{(2)}\g\|_{B^{s}_{p,\infty}}
&\geq2^{{ns}}
\f\|\De_{n}\f(\mathbf{S}_{t}(u_0)-u_0\g)^{(2)}\g\|_{L^p}
=2^{{ns}}\f\|\De_{n}\big(t\mathbf{v}_0(u_0)-\mathbf{w}\big)^{(2)}\g\|_{L^p}
\nonumber\\
&\geq t2^{{n}s}\f\|\De_{n}\big(u_0\cd\na u_0^{(2)}\big)\g\|_{L^p}-
t2^{{n}s}\f\|\De_{n}\mathcal{Q}(u_0\cd\na u_0)\g\|_{L^p}\nonumber\\
&\quad- t2^{{n}s}\f\|\De_{n}\big(\Omega\mathcal{P}(e_3\times u_0))\big)\g\|_{L^p}-C2^{2{n}}\|\mathbf{w}\|_{\dot{B}^{s-2}_{p,\infty}}
\nonumber\\
&\geq ct2^{n}-Ct-C2^{2n}t^2.
\end{align*}
Thus, picking $t2^{kn}\approx\ep$ with small $\ep$, we have
\bbal
\f\|\f(\mathbf{S}_{t}(u_0)-u_0\g)^{(2)}\g\|_{B^{s}_{p,\infty}}\geq c\ep-C\ep^2\geq \tilde{c}\ep.
\end{align*}
This completes the proof of Theorem \ref{th4}.

\section{Discussion and Conclusion }
In this paper, we consider the Cauchy problem for the 3D Euler equations with the Coriolis force in the whole space. We first establish the local-in-time existence and uniqueness of solution to this system in $B^s_{p,r}(\R^3)$. Then we prove that the Cauchy problem is ill-posed in two different sense:
(1) the solution of this system is not uniformly continuous dependence on the initial data in the same Besov spaces, which extends the recent work of Himonas-Misio{\l}ek \cite[Comm. Math. Phys., 296, 2010]{HM1} to the more general framework of Besov spaces;
(2) the solution of this system cannot be H\"{o}lder continuous in time variable in the same Besov spaces. In particular, the solution of the system is discontinuous in the weaker Besov spaces at time zero. To the best of our knowledge, our work is the first one addressing the issue on the failure of H\"{o}lder continuous in time of solution to the classical Euler equations with(out) the Coriolis force. In this paper, our contribution on the continuous properties for the 3D incompressible rotating Euler equations, unfold in three aspects: by constructing some proper, new and different initial conditions, we showed the non-uniform continuous of the solution map (Theorem \ref{th2}), proved the failure of H\"{o}lder regularity of the solution map (Theorem \ref{th3}) and  established the discontinuous at $t = 0$ of the solution map (Theorem \ref{th4}).

{\bf Future directions.} In the case when $p\in(1,\infty)$, since the Coriolis force term has no impact on the well-posedness results, we can prove Theorem \ref{th1} by following the proof of the local well-posedness results for the Euler equations (see \cite[Theorem 7.1 ]{B} and \cite[Theorem 1.1 ]{GLY}). However, in the case when $p=1$ or $p=\infty$, due to the the appearance of the dispersive effect of rotation which leads to the fact that Riesz transform does not map continuously
from $L^\infty(\text{or}\;L^1)$ to $L^\infty(\text{or}\;L^1)$, we will encounter the main difficulty when establishing the uniform bounds of solution in $L^\infty(\text{or}\;L^1)$-based spaces. We believe that, the dispersive effect of rotation is able to prevent well-posedness for the Cauchy problem \eqref{CE}. More precisely, we expect that the Cauchy problem \eqref{CE} is ill-posed in $L^\infty(\text{or}\;L^1)$-based Besov spaces. We will investigate this interesting problem in the future work.

\section*{Declarations}

\noindent\textbf{Availability of data and materials}\\ No data was used for the research described in the article.
\vspace*{1em}

\noindent\textbf{Conflict of interest}\\
The authors declare that they have no conflict of interest.
\vspace*{1em}

\noindent\textbf{Funding}\\
The authors would like to thank the anonymous referees for valuable comments and suggestions which greatly improved the presentation of this paper. J. Li is supported by the National Natural Science Foundation of China (No. 12161004), Innovative High end Talent Project in Ganpo Talent Program (No. gpyc20240069), Training Program for Academic and Technical Leaders of Major Disciplines in Ganpo Juncai Support Program (No. 20232BCJ23009), Jiangxi Provincial Natural Science Foundation (No. 20252BAC210004).
N. Zhu is partially supported by the National Natural Science Foundation of China (No. 12161055), Training Program for Academic and Technical Leaders of Major Disciplines in Ganpo Juncai Support Program (No. 20243BCE51079).

\end{document}